\documentclass[a4paper]{article}
\usepackage[utf8]{inputenc}

\usepackage[top=2.5cm, bottom=2.5cm, left=2.5cm, right=2.5cm]{geometry}

\usepackage{amsmath, amssymb, tcolorbox}
\usepackage[shortlabels]{enumitem}
\usepackage{url}

\usepackage{tikz}
\usepackage{soul}

 \usepackage{amsthm}
 \newtheorem{thm}{Theorem}[section]
 \newtheorem{prop}[thm]{Proposition}
 \newtheorem{lm}[thm]{Lemma}
 \newtheorem{res}[thm]{Result}
 \newtheorem{crl}[thm]{Corollary}

 \theoremstyle{definition}
 \newtheorem{rmk}[thm]{Remark}
 \newtheorem{df}[thm]{Definition}
 \newtheorem{constr}[thm]{Construction}
 \newtheorem{que}[thm]{Question}

 \newcommand{\vspan}[1]{\left \langle #1 \right \rangle}
 \newcommand{\set}[1]{ \left \{ #1 \right \} }
 \newcommand{\sett}[2]{ \left\{ #1 \, \, || \, \, #2 \right \} }

 \newcommand{\gauss}[2]{\genfrac{[}{]}{0pt}{}{#1}{#2}_q}
 \newcommand{\gaussTwo}[2]{\genfrac{[}{]}{0pt}{}{#1}{#2}_2}
 \newcommand{\pg}{\textnormal{PG}}
 
\renewcommand{\geq}{\geqslant}
\renewcommand{\leq}{\leqslant}

\title{Blocking subspaces with points and hyperplanes}
\author{Sam Adriaensen \\ \textit{Vrije Universiteit Brussel} \and Maarten De Boeck \\ \textit{University of Memphis} \and Lins Denaux \\ \textit{Ghent University}}
\date{}

\begin{document}

\maketitle

\begin{abstract}
In this paper, we characterise the smallest sets $B$ consisting of points and hyperplanes in $\pg(n,q)$, such that each $k$-space is incident with at least one element of $B$.
If $k > \frac {n-1}2$, then the smallest construction consists only of points.
Dually, if $k < \frac{n-1}2$, the smallest example consists only of hyperplanes.
However, if $k = \frac{n-1}2$, then there exist sets containing both points and hyperplanes, which are smaller than any blocking set containing only points or only hyperplanes.
\end{abstract}



\begin{tcolorbox}[colback=red!5!white,colframe=red!75!black,title=Update]
 After publication of this paper, we discovered that in case $k = \frac{n-1}2$, the correct lower bound and a classification of the smallest examples was already established by Blokhuis, Brouwer, and Sz\H onyi \cite[Proposition 5.1]{Blokhuis_Brouwer:12}.
\end{tcolorbox}

\section{Introduction}

Throughout this article, $q$ will denote a prime power, $k$ and $n$ will be integers satisfying $0 \leq k < n$, and $\pg(n,q)$ will denote the $n$-dimensional projective space over the finite field with $q$ elements.
When we mention dimensions, we always mean projective dimensions.

Blocking sets in projective spaces are classically defined as follows.

\begin{df}
A set of points $B$ in $\pg(n,q)$ is called a \emph{blocking set} with respect to $k$-spaces if every $k$-space of $\pg(n,q)$ is incident with at least one element of $B$.
\end{df}

The main objective in the study of blocking sets is the characterisation of blocking sets of small size. For a survey, see \cite{blokhuissziklaiszonyi}. A classical result by Bose and Burton characterises the smallest blocking sets.

\begin{prop}[{\cite{boseburton}}]
 \label{BoseBurtonProp}
    The smallest blocking sets with respect to $k$-spaces in $\pg(n,q)$ are precisely the point sets of $(n-k)$-spaces.
\end{prop}

A \emph{duality} of a projective geometry $\pg(n,q)$ is a bijective map $\perp$ on the subspaces of $\pg(n,q)$ that reverses inclusion, with the property that $\dim \pi + \dim \pi^\perp = n-1$ for any subspace $\pi$ (see \cite[pp.\ 78-79, 88]{kissszonyi}). The image of a subspace under a duality is called its \emph{dual}. In particular, the dual of a point is a hyperplane and vice versa.

We combine the notion of a blocking set and its dual to define the following generalised notion of a blocking set.

\begin{df}
Let $B$ be a set of points and hyperplanes in $\pg(n,q)$.
We call $B$ a \emph{blocking set} with respect to $k$-spaces if every $k$-space of $\pg(n,q)$ is incident with at least one element of $B$. That is, for every $k$-space $\kappa$, $B$ contains a point of $\kappa$ or a hyperplane through $\kappa$.
\end{df}

Throughout this work, if $B$ is a blocking set, we will denote the set of points in $B$ by $B_0$ and the set of hyperplanes in $B$ by $B_{n-1}$, hence $B=B_0\cup B_{n-1}$.
If $B$ contains only points, we will emphasise this by calling $B$ a \emph{blocking set of points}.

In \cite[Proposition 7.1]{debeuledhaeseleerihringermannaert}, the authors give a lower bound on the number of variables on which a certain degree 2 Boolean function depends.
The lower bound follows from a lower bound on the size of a blocking set of points and hyperplanes with respect to 4-spaces in $\pg(7,q)$.

We will answer the following more general question.

\begin{que}
 \label{Question}
What are the smallest blocking sets consisting of point and hyperplanes in $\pg(n,q)$ with respect to $k$-spaces?
\end{que}

To answer this question, we introduce the following construction, which generalises an idea of Sam Mattheus.

\begin{constr}
 \label{Constr}
Consider a $(k+1)$-space $\Sigma$ and a $(k-1)$-subspace $\sigma\subset\Sigma$ in $\pg(2k+1,q)$.
Partition the set of $k$-spaces in $\Sigma$ through $\sigma$ into two non-empty sets $K_1$ and $K_2$. Note that $|K_1|+|K_2|=q+1$. Define the sets
\begin{itemize}
 \item $B_0$ as the set of all points that are contained in some $\kappa \in K_1$, but not in $\sigma$, and
 \item $B_{n-1}$ as the set of all hyperplanes that go through some $\kappa \in K_2$, but not through $\Sigma$.
\end{itemize}
Define $B := B_0 \cup B_{n-1}$.
\end{constr}

\begin{lm}
Let $B$ be as in Construction \ref{Constr}.
Then 
\begin{enumerate}[(1)]
 \item $B$ blocks all $k$-spaces of $\pg(2k+1,q)$,
 \item $|B| = (q+1)q^k$, and
 \item the dual of $B$ is also an instance of Construction \ref{Constr}.
\end{enumerate}
\end{lm}

\begin{proof}
\begin{enumerate}[(1)]
\item Consider a $k$-space $\tau$. Note that by Grassmann's identity, $\tau \cap \Sigma$ is non-empty. If $\tau \cap \Sigma\subseteq \kappa$ for a $\kappa\in K_{2}$, then at least one of the hyperplanes through $\vspan{\tau,\kappa}$ does not go through $\Sigma$ and hence is contained in $B_{n-1}$; this includes the case of $\tau \cap \Sigma\subseteq \sigma$. If $\tau \cap \Sigma\subseteq \kappa$ for a $\kappa\in K_{1}$, but $\tau \cap \Sigma\not\subseteq \sigma$, then $\tau$ contains a point of $B_{0}$. If $\tau \cap \Sigma$ is not contained in a $k$-space of $K_1\cup K_2$, then $\tau \cap \Sigma$ contains a line $\ell$ disjoint from $\sigma$. Such a line contains a point of $B_0$ since $K_1$ is non-empty. Hence, in each case, $\tau$ is blocked by an element of $B$.

\item We immediately find $|B| = |B_0| + |B_{n-1}| = |K_1|q^k + |K_2| q^k = (q+1)q^k$.
\item This can be seen directly from the description of the construction. \qedhere
\end{enumerate}
\end{proof}

We are now ready to state the result of this paper.

\begin{thm}
 \label{ThmMain}
Let $q$ be a prime power, and let $k,n$ be integers with $0 \leq k < n$. If $q=2$, suppose that $k \notin \set{\frac{n-2}2, \frac n 2}$. Let $B$ be a blocking set with respect to $k$-spaces in $\pg(n,q)$ consisting of points and hyperplanes.
\begin{enumerate}[(1)]
    \item If $k < \frac{n-1}2$, then 
    \(|B| \geq \frac{q^{k+2}-1}{q-1}\), with equality if and only if $B$ consists of all hyperplanes through a fixed $(n-k-2)$-space.
    \item If $k > \frac{n-1}2$, then
    \(
    |B| \geq \frac{q^{n-k+1}-1}{q-1},
    \)
    with equality if and only if $B$ consists of all points in a fixed $(n-k)$-space.
    \item If $k = \frac{n-1}2$, then $|B| \geq (q+1)q^k$, with equality if and only if $B$ is as in Construction \ref{Constr}.
\end{enumerate}
\end{thm}

The first two parts of this theorem will be proven in Section \ref{SecNotMiddle}, the last part in Section \ref{SecMiddle}.
First, we will discuss some background on blocking sets of points.

\section{Preliminaries}

The number of $m$-spaces in $\pg(n,q)$ is given by the Gaussian coefficient
\[
 \gauss{n+1}{m+1} := \prod_{i=1}^{m+1} \frac{q^{n-m+i}-1}{q^i-1},
\]
see e.g.\ \cite[Theorem 4.7]{kissszonyi}. In particular, we set $\gauss{a}{b}:=0$ unless $0\leq b\leq a$. We will also use the notation
\[
 \theta_m := \gauss{m+1}{1} = \frac{q^{m+1}-1}{q-1} = q^m + q^{m-1} + \dots + q + 1
\]
for the number of points in $\pg(m,q)$.

An important result related to blocking sets is the following theorem by Metsch \cite{metsch2006}.

\begin{res}[{\cite[Theorem 1.2]{metsch2006}}]\label{ResKlaus}
Let $d$, $s$ and $n$ be integers with $d, s \geq 0$ and $n \geq d+s$, and consider a point set $B$ in $\pg(n,q)$.
If $|B| \leq \theta_d$, then the number of $s$-spaces not containing a point of $B$ is at least
 \[
    q^{(s+1)(d+1)} \gauss{n-d}{s+1} + ( \theta_d - |B| ) q^{sd} \gauss{n-d}s.
 \]
\end{res}

The result by Metsch will be crucial to prove Theorem \ref{ThmMain} when $k \neq \frac{n-1}2$.
By applying the principle of duality, this theorem has the following corollary.

\begin{crl}\label{CrlKlausDual}
Let $d$, $s$ and $n$ be integers with $d \geq 0$ and $d-1\leq s < n$, and consider a set $B$ of hyperplanes in $\pg(n,q)$. If $|B| \leq \theta_{d}$, then the number of $s$-spaces not contained in an element of $B$ is at least
\[
    q^{(n-s)(d+1)} \gauss{n-d}{n-s} + (\theta_{d}-|B|)q^{(n-s-1)d} \gauss{n-d}{n-s-1}.
\]
\end{crl}

The next result by Héger and Nagy gives an upper bound on a Gaussian binomial.

\begin{res}[{\cite[Lemma $2.2$]{hegernagy2022}}]\label{ResTamasZoli}
With \( e \) being Euler's number, we have:
    \begin{align*}
     \gauss n k < \begin{cases}
      q^{(n-k)k}\cdot e^\frac{1}{q-2} & \text{if } q > 2, \\
      2^{(n-k)k+1}\cdot e^\frac{2}{3} & \text{if } q = 2.
     \end{cases}
    \end{align*}
\end{res}

We would like to remark that Result \ref{ResKlaus} implies the result by Bose and Burton (see Proposition \ref{BoseBurtonProp}).
The latter was strengthened by Bruen \cite{bruen71} for blocking sets of points with respect to lines in $\pg(2,q)$.
This in turn was generalised by Beutelspacher \cite{beutelspacher80}.

\begin{res}[{\cite{beutelspacher80}}]
 \label{ResBeutelspacher}
Let $B$ be a blocking set of points with respect to $k$-spaces in $\pg(n,q)$.
Then either
\begin{enumerate}[(1)]
 \item $B$ contains all points of a fixed $(n-k)$-space and
 \(
  |B| \geq \theta_{n-k},\text{ or}
 \)
 \item $B$ does not contain all points of a fixed $(n-k)$-space and
 \(
 |B| \geq \theta_{n-k} + q^{n-k-1} \sqrt q.
 \)
\end{enumerate}
\end{res}

\section{The case \(\displaystyle k\neq\frac{n-1}2\)}
 \label{SecNotMiddle}

\begin{thm}
    Let $k\neq\frac{n-1}2$ and, if $q=2$, suppose that $k \notin \set{\frac{n-2}2, \frac n 2}$.
    Let $B$ be a blocking set with respect to $k$-spaces in $\pg(n,q)$ consisting of points and hyperplanes.
    Then $|B|\geq\min\set{\theta_{k+1},\theta_{n-k}}$ with equality if and only if
    \begin{enumerate}[(1)]
        \item $k < \frac{n-1}2$ and $B$ consists of all hyperplanes through a fixed $(n-k-2)$-space, so $|B|=\theta_{k+1}$, or
        \item $k > \frac{n-1}2$ and $B$ consists of all points in a fixed $(n-k)$-space, so $|B|=\theta_{n-k}$.
    \end{enumerate}
\end{thm}
\begin{proof}
    Recall that $B_0$ and $B_{n-1}$ are the set of points and the set of hyperplanes in $B$, respectively. We may suppose that $k<\frac{n-1}{2}$ as (2) follows from (1) using duality. Moreover, we assume that $|B|\leq\min\set{\theta_{k+1},\theta_{n-k}}=\theta_{k+1}$ and will prove that $B_0$ is empty; (the dual of) Proposition \ref{BoseBurtonProp} then finishes the proof.
    As $|B_{n-1}|\leq|B|\leq\theta_{k+1}$, we can apply Corollary \ref{CrlKlausDual} (with $d:=k+1$ and $s:=k$) to obtain that $B_{n-1}$ does not block at least
    \[
        (\theta_{k+1}-|B_{n-1}|)q^{(n-k-1)(k+1)}\geq|B_0|q^{(n-k-1)(k+1)}
    \]
    $k$-spaces, which means that the points in $B_0$ have to block at least $|B_0|q^{(n-k-1)(k+1)}$ $k$-spaces.
    Any point of $\pg(n,q)$ is incident with $\gauss{n}{k}$ $k$-spaces.
    If $q > 2$, Result \ref{ResTamasZoli} implies
    \[
        \gauss{n}{k}<q^{(n-k)k}\cdot e^\frac{1}{q-2}=q^{(n-k-1)(k+1)}\frac{e^\frac{1}{q-2}}{q^{n-2k-1}}\leq q^{(n-k-1)(k+1)},
    \]
    where we used the fact that $e^\frac{1}{q-2}\leq q^{n-2k-1}$ if and only if $1\leq(q-2)(n-2k-1)\ln(q)$, the latter of which is always true as $q\geq3$ and $k\leq\frac{n-2}{2}$.
    If $q=2$, any point of $\pg(n,q)$ is incident with at most
    \[
        \gaussTwo{n}{k}<2^{(n-k)k+1}\cdot e^\frac{2}{3}=2^{(n-k-1)(k+1)}\frac{e^\frac{2}{3}}{2^{n-2k-2}}\leq 2^{(n-k-1)(k+1)}
    \]
    $k$-spaces, as $e^\frac{2}{3}\leq 2^{n-2k-2}$ if and only if $\frac{2}{3\ln(2)}\leq n-2k-2$, the latter of which is true if $k\leq\frac{n-3}{2}$.
    In conclusion, we obtain that any point is incident with less than $q^{(n-k-1)(k+1)}$ $k$-spaces, which contradicts the statement that the points in $B_0$ have to block at least $|B_0|q^{(n-k-1)(k+1)}$ $k$-spaces unless $B_0$ is empty.
\end{proof}

\begin{rmk}
    It remains open whether the same holds for the case of $n$ even, $k\in\set{\frac{n-2}{2},\frac{n}{2}}$, $q=2$.
    
    However, it is not that surprising that this case does not submit itself to mere approximations.
    For $k=\frac{n}{2}$ and $q=2$, we can define $B_0$ to be the set of points in an $\frac{n}{2}$-space $\Sigma$ except for the points in some $(\frac{n}{2}-1)$-space $\kappa\subset\Sigma$, and $B_{n-1}$ to be the set of hyperplanes through $\kappa$ not containing $\Sigma$ (i.e.\ similar to Construction \ref{Constr} with $\Sigma$ an $\frac n2$-space and $\sigma$ an $\left(\frac n2 - 2 \right)$-space in $\pg(n,q)$, whereby $\sigma\subset\kappa$, $|K_1|=2$ and $|K_2|=1$). The set $B=B_0\cup B_{n-1}$ blocks all $\frac{n}{2}$-spaces and has size $2^{\frac{n}{2}+1}$, while a trivial blocking set consisting of all points in an $\frac{n}{2}$-space has size $2^{\frac{n}{2}+1}-1$.
    
    A similar observation exists for the dual case $k=\frac{n-2}{2}$.
\end{rmk}

\section{The case $\displaystyle k = \frac{n-1}2$}
 \label{SecMiddle}

Assume that $n=2k+1\geq 3$ and consider a blocking set $B=B_0 \cup B_{n-1}$ in $\pg(n,q)$ with respect to $k$-spaces.

\begin{lm}\label{Lm_ElementaryLemma}
Suppose that $\rho$ is a $(k-1)$-space skew to $B_0$.
Then $\rho$ lies in at least $q+1-\frac{|B_0|}{q^k}$ hyperplanes of $B_{n-1}$.
If equality holds, then 
\begin{enumerate}[(1)]
    \item every $k$-space through $\rho$ contains at most one point of $B_0$, and
    \item $|B_0|$ is a multiple of $q^k$.
\end{enumerate}
\end{lm}

\begin{proof}
We project from $\rho$ as follows: we take a $(k+1)$-space $\kappa$ skew to $\rho$ and define
\begin{align*}
 B_0' = \sett{ \vspan{\rho,P} \cap \kappa }{P \in B_0},&&
 B_{n-1}' = \sett{ \pi \cap \kappa}{\rho \subset \pi \in B_{n-1}}.
\end{align*}
Then $B_0' \cup B_{n-1}'$ needs to block all the points in $\kappa$, since $B_0 \cup B_{n-1}$ blocks all $k$-spaces through $\rho$.
We may assume that $|B_{n-1}'|\leq q+1$ since otherwise the lemma is trivially fulfilled.
$B_0'$ needs to contain all points of $\kappa$ not covered by the $k$-spaces in $B_{n-1}'$.
Applying Corollary \ref{CrlKlausDual} in $\kappa\cong\pg(k+1,q)$ (with $d:=1$ and $s:=0$) we find that 
\[
    |B_{0}'|\geq (q+1-|B_{n-1}'|)q^{k}.
\]
This implies that
\[
 |B_{n-1}'| \geq q+1 -\frac{|B_0'|}{q^k}
 \geq q+1 - \frac{|B_0|}{q^k}.
\]
If this bound is tight, then $|B_0'| = |B_0|$, which means that $\vspan{P,\rho} \neq \vspan{R,\rho}$ for any two points $P$ and $R$ of $B_0$.
Thus, no $k$-space through $\rho$ contains more than one point of $B_0$.
Obviously, equality also implies that $q+1-\frac{|B_0|}{q^k}$ is integer, so $q^k$ must divide $|B_0|$.
\end{proof}

\begin{lm}\label{Lm_ElementaryTheorem}
\(
 |B| \geq q^k(q+1).
\)
If equality holds, then the following is true.
\begin{enumerate}[(1)]
 \item Let $\rho$ be a $(k-1)$-space skew to $B_0$.
 If there is a $k$-space through $\rho$ that is incident with exactly one point of $B_0$ and no hyperplane of $B_{n-1}$, then every $k$-space through $\rho$ contains at most one point of $B_0$.
 \item No hyperplane of $B_{n-1}$ contains a point of $B_0$.
 \item $|B_0|$ is a multiple of $q^k$.
\end{enumerate}
\end{lm}
\begin{proof}
It suffices to prove the lemma in the case that $B$ is a  \emph{minimal} blocking set, i.e.\ no proper subset of $B$ is a blocking set.
If the lemma is true for minimal blocking sets, the size of a non-minimal blocking set is greater than $q^k(q+1)$, and the lemma is true for all blocking sets.

If $B_0 = \varnothing$, then, by (the dual of) Proposition \ref{BoseBurtonProp}, $B_{n-1}$ contains at least $\theta_{k+1} > (q+1)q^k$ hyperplanes hence $|B| > (q+1)q^k$.
So we may assume that $B_0 \neq \varnothing$, and take any point $P \in B_0$.
There exists a $k$-space $\kappa$ that is incident with $P$ and no other element of $B$.
Otherwise, $B \setminus \set P$ would also be a blocking set, contradicting the minimality of $B$.
In $\kappa$ there are $q^k$ distinct $(k-1)$-spaces missing $P$ and, by Lemma \ref{Lm_ElementaryLemma}, each of these $(k-1)$-spaces must be contained in at least $q+1 - \frac{|B_0|}{q^k}$ hyperplanes of $B_{n-1}$.
Since no hyperplane of $B_{n-1}$ contains $\kappa$, no hyperplane is counted twice.
This implies that
\[
 |B_{n-1}| \geq q^k \left( q+1 - \frac{|B_0|}{q^k} \right) \quad\iff\quad |B_0|+|B_{n-1}|\geq q^k \left( q+1 \right).
\]
If equality holds in this bound, then equality holds for the bound described in Lemma \ref{Lm_ElementaryLemma}, for every $(k-1)$-space of $\kappa$ that misses $P$. Therefore, by that same lemma, no $(k-1)$-space $\rho \subset \kappa$ that misses $P$ lies in a $k$-space containing more than 1 point of $B_0$, and $|B_0|$ is a multiple of $q^k$. Moreover, in the above argument, we did not count any hyperplane through $P$. Thus, if equality holds, no point of $B_0$ lies in a hyperplane of $B_{n-1}$.
\end{proof}

\begin{df}
Let $S$ be a set of points in $\pg(n,q)$.
We call a line $\ell$ \emph{skew, tangent, or secant} to $S$ if $\ell$ intersects $S$ in respectively 0, 1, or at least 2 points.
\end{df}

\begin{lm}
 \label{LmOnlyTangentToB0}
 If $|B| = (q+1)q^k$, then no point $P \notin B_0$ lies on both tangent and secant lines to $B_0$.
\end{lm}
\begin{proof}
 Suppose that $P \notin B_0$ lies on a tangent line $\ell$ to $B_0$.
 First, we prove that $\ell$ lies in some $k$-space which contains only one point of $B_0$.
 If $k=1$, $\ell$ itself is a $k$-space, and we are done, so suppose that $k > 1$.
 Take an $(n-2)$-space $\Pi$ disjoint to $\ell$.
 Consider the projection of $B_0 \setminus \ell$ from $\ell$ onto $\Pi$, i.e.\ the point set
 \[
  B_0' := \sett{\vspan{\ell,R} \cap \Pi}{R \in B_0 \setminus \ell}.
 \]
 Then $|B_0'| < (q+1)q^k$. Therefore, by Proposition \ref{BoseBurtonProp}, $B_0'$ cannot block all $(k-2)$-spaces of $\Pi$. Hence, we can take a $(k-2)$-space $\tau$ in $\Pi$, skew to $B_0'$.
 Then $\kappa:=\vspan{\tau,\ell}$ is a $k$-space through $\ell$, containing only one point of $B_0$, and hence is not contained in a hyperplane of $B_{n-1}$ by Lemma \ref{Lm_ElementaryTheorem}(2).
 Now suppose that some line $\ell_2$ through $P$ contains at least two points of $B_0$.
 Take a $(k-1)$-space $\rho$ in $\kappa$ through $P$, missing $B_0$.
 Then $\vspan{\rho,\ell_2}$ is a $k$-space through $\rho$ containing at least two points of $B_0$, contradicting Lemma \ref{Lm_ElementaryTheorem}(1).
\end{proof}

\begin{lm}
 \label{LmOnlyTangentToS}
Let $S$ be a non-empty point set of $\pg(n,q)$ such that no point $P \notin S$ lies on both tangent and secant lines to $S$.
Then
\[
 \Sigma := S \cup \sett{P \notin S}{P \textnormal{ does not lie on a tangent to } S}
\]
is a subspace of dimension 
\(
 \min \sett{m}{|S| \leq \theta_m}
\).
\end{lm}

\begin{proof}
Take a point $P \notin \Sigma$. First, take a plane $\pi$ through $P$ that contains the points $R_1, R_2, R_3 \in S$.
Suppose that these three points are not collinear.
Then $\vspan{R_1,R_2} \cap \vspan{P,R_3}$ is a point distinct from $R_3$.
This point cannot be contained in $S$, because $\vspan{P,R_3}$ must be tangent to $S$ in $R_3$.
But it lies both on the tangent line $\vspan{P,R_3}$ to $S$, and the secant line $\vspan{R_1,R_2}$ to $S$.
This contradicts our assumption on $S$.
Hence, $\pi \cap S$ must be contained in a line.

Now suppose that some line $\ell$ through $P$ contains at least two points $R_1, R_2 \in \Sigma$.
Take a point $R_3 \in S \setminus \ell$
(if $S \subset \ell$, then $S$ contains at most a point, and the lemma is proven).
Then $\vspan{R_1,R_3}$ and $\vspan{R_2,R_3}$ are secant lines to $S$.
But then the intersection of the plane $\vspan{\ell,R_3}$ with $S$ is not contained in a line, contradicting the first part of the proof.
Hence, any point $P \notin \Sigma$ only lies on skew and tangent lines to $\Sigma$.
Therefore, $\Sigma$ is a subspace.

It is clear that 
\(\dim \Sigma \geq \min \sett m {|S| \leq \theta_m} \).
Now suppose that $|S| \leq \theta_m$ for some $m$, but that $\dim \Sigma > m$.
Take a point $P \in S$.
No line through $P$ in $\Sigma$ is tangent to $S$.
Therefore, there exist at least $\theta_m$ lines of $\Sigma$ through $P$ that contain at least one other point of $S$ and
\[
 |S| \geq 1 + \theta_m,
\]
a contradiction.
\end{proof}

\begin{lm}
 \label{LmBP}
Assume that $B_0$ is contained in a $(k+1)$-space $\Sigma$ and consider a point $P \in \Sigma \setminus B_0$.
Define
\[
 B_P := \sett{\pi \in B_{n-1}}{ P \in \pi \not \supseteq \Sigma }.
\]
Then one of the following holds.
\begin{enumerate}[(1)]
 \item There exists a $k$-space $\rho \subset \Sigma$ such that $B_P$ contains all hyperplanes $\pi$ with $\pi \cap \Sigma = \rho$, which implies $|B_P| \geq q^k$.
 \item \( |B_P| \geq q^{k-1}(q+1) \).
\end{enumerate}
\end{lm}
\begin{proof}
$B_P$ must block all $k$-spaces intersecting $\Sigma$ exactly in the point $P$, since no hyperplane through $\Sigma$ contains a $k$-space intersecting $\Sigma$ only in $P$.
Consider the quotient geometry through $P$, which is isomorphic to $\pg(n-1,q)$.
Then \(B_P / P := \sett{\pi / P}{\pi \in B_P}\) is a set of hyperplanes in $\pg(n-1,q)$ blocking all $(k-1)$-spaces skew to the $k$-space $\Sigma/P$.
Choose a duality $\perp$ of $\pg(n-1,q)$.
Then \((B_P/P)^\perp := \sett{(\pi/P)^\perp}{\pi \in B_P}\) is a set of points in $\pg(n-1,q)$ blocking all $k$-spaces skew to the $(k-1)$-space $(\Sigma/P)^\perp$.
Therefore, the union of 
\( (B_P/P)^\perp \)
and the points of
\( (\Sigma/P)^\perp \)
is a blocking set with respect to $k$-spaces in $\pg(2k,q)$.

First, suppose that this blocking set contains a $k$-space.
This space is of the form $(\rho/P)^\perp$ for some $k$-space $\rho$ in $\pg(n,q)$ through $P$.
If $(\Sigma/P)^\perp \subset (\rho/P)^\perp$, i.e.\ $\rho \subset \Sigma$, then
\(
 |B_P| \geq \theta_k - \theta_{k-1}.
\)
Otherwise, $(\Sigma/P)^\perp$ and $(\rho/P)^\perp$ intersect in at most a $(k-2)$-space and
\(
 |B_P| \geq \theta_k - \theta_{k-2}.
\)

Now suppose that the blocking set does not contain a $k$-space.
Then, by Result \ref{ResBeutelspacher}, it contains at least $\theta_k + q^{k-1} \sqrt q$ points, which implies that
\(
 |B_P| \geq \theta_k + q^{k-1} \sqrt q - \theta_{k-1}.
\)
\end{proof}

\begin{thm}
$|B| \geq (q+1)q^k$, with equality if and only if $B$ is as in Construction \ref{Constr}.
\end{thm}
\begin{proof}
We already proved that $|B| \geq (q+1)q^k$ in Lemma \ref{Lm_ElementaryTheorem}, so suppose that equality holds.
By Lemmas \ref{LmOnlyTangentToB0} and \ref{LmOnlyTangentToS}, $B_0$ lies in a $(k+1)$-space $\Sigma$.
By duality, there must also exist a $(k-1)$-space $\sigma$ that lies in all hyperplanes of $B_{n-1}$.
By Lemma \ref{Lm_ElementaryTheorem}(2), every hyperplane of $B_{n-1}$ intersects $\Sigma$ in a $k$-space skew to $B_0$.

Suppose that $\sigma \not \subset \Sigma$.
Let $n_i$ denote the number of points of $\Sigma \setminus B_0$ that are contained in exactly $i$ hyperplanes of $B_{n-1}$. Given a $k$-space $\rho \subset \Sigma$, there exists a hyperplane which intersects $\Sigma$ in $\rho$, and which does not contain $\sigma$. Therefore, Lemma \ref{LmBP}(1) does not hold, and $n_i = 0$ for all $i < q^{k-1}(q+1)$. This implies that
\begin{align*}
 |B_{n-1}| \theta_k
 = \sum_i i n_i
 \geq q^{k-1}(q+1) \sum_i n_i
 = q^{k-1}(q+1) (\theta_{k+1} - |B_0|).
\end{align*}
Then the above inequality implies that
\begin{align*}
 |B_{n-1}| \theta_k \geq q^{k-1}(q+1) (\theta_{k+1} - |B_0|)
 & = q^{k-1}(q+1) (\theta_{k+1} - (q^{k+1} + q^k - |B_{n-1}|)) \\
 & = q^{k-1}(q+1) (|B_{n-1}| + \theta_{k-1}).
\end{align*}
If $k=1$, this simplifies to
\(
 |B_{n-1}| (q+1) \geq (q+1) (|B_{n-1}| + 1),
\)
a contradiction. If $k > 1$, we can rewrite this inequality as
\[
 |B_{n-1}| \geq q^{k-1} (q+1) \frac{q^k-1}{q^{k-1}-1}
 > q^{k-1}(q+1)q,
\]
which contradicts $|B_0 \cup B_{n-1}| = q^k(q+1)$.

Hence, $\sigma \subset \Sigma$.
We know from Lemma \ref{Lm_ElementaryTheorem}(3) that $|B_0| = t q^k$ for some $t$, and therefore $|B_{n-1}| = (q+1-t) q^k$.
By Lemma \ref{Lm_ElementaryTheorem}(2), every hyperplane of $B_{n-1}$ intersects $\Sigma$ in a $k$-space through $\sigma$.
Consider the set
\(
 S := \sett{\pi \cap \Sigma}{\pi \in B_{n-1}}.
\)
Since each $k$-space $\kappa$ of $\Sigma$ lies in $q^k$ hyperplanes intersecting $\Sigma$ in $\kappa$, $|S| \geq q+1-t$, with equality if and only if $B_{n-1}$ consists of all hyperplanes intersecting $\Sigma$ in a $k$-space of $S$.
On the other hand, by Lemma \ref{Lm_ElementaryTheorem}(2), no point of $B_0$ is contained in $U=\cup_{\kappa\in S}\kappa$, hence
\[
 t q^k = |B_0| \leq | \Sigma \setminus U |
 = \theta_{k+1} - (|S| q^k + \theta_{k-1})
 = (q+1-|S|)q^k.
\]
This implies that \(|S| \leq q+1-t\), with equality if and only if $B_0 = \Sigma \setminus U$.

In conclusion, let $T$ denote the set of the $t$ $k$-spaces in $\Sigma$ through $\sigma$, not contained in $S$.
Then
\begin{itemize}
 \item $B_{n-1}$ consists of all hyperplanes intersecting $\Sigma$ exactly in an element of $S$, and
 \item $B_0$ consists of all points of $\kappa \setminus \sigma$, where $\kappa$ varies over the elements $T$.
\end{itemize}
Thus, $B$ is as in Construction \ref{Constr}.
\end{proof}

\textbf{Acknowledgements.}
This problem was proposed at a brainstorm session by Ferdinand Ihringer. We would like to thank everyone present at the brainstorm session for helpful discussions, especially Sam Mattheus. The first author would also like to thank the Faculty of Mathematics at the University of Rijeka for their hospitality during his stay. This work was partially supported by Croatian Science Foundation under the project 5713.

\bigskip

\begin{center}
\begin{minipage}{.45 \textwidth}
Sam Adriaensen \\
\textit{Vrije Universiteit Brussel} \\
Pleinlaan 2, 1050 Elsene, Belgium \\
\url{sam.adriaensen@vub.be}
\end{minipage}
\begin{minipage}{.45 \textwidth}
Maarten De Boeck \\
\textit{University of Memphis} \\
Dunn Hall, Norriswood Ave, Memphis, TN 38111, United States \\
\url{mdeboeck@memphis.edu}
\end{minipage}

\bigskip

\begin{minipage}{.45 \textwidth}
Lins Denaux \\
\textit{Ghent University} \\
Krijgslaan 281, 9000 Gent, Belgium \\
\url{lins.denaux@ugent.be}
\end{minipage}

\end{center}

\end{document}